\documentclass[11pt,a4paper]{amsart}

\usepackage{amsfonts,amsmath, amssymb}
\usepackage{color}

\usepackage{mathpazo}
\usepackage{geometry}                % See geometry.pdf to learn the layout options. There are lots.
\usepackage{anysize}
\marginsize{3cm}{3cm}{3cm}{3cm}%\marginsize{left}{right}{top}{bottom}

\renewcommand{\baselinestretch}{1.25}
\usepackage[sort&compress,square,comma,numbers]{natbib}

\usepackage{graphicx,float,enumerate}
\usepackage[ruled,boxed,commentsnumbered,norelsize]{algorithm2e}
\usepackage{verbatim}
\usepackage{url}
\usepackage{epstopdf}
\usepackage[capitalise]{cleveref}

\newtheorem{theorem}{Theorem}[section]
\newtheorem{corollary}[theorem]{Corollary}
\newtheorem{lemma}[theorem]{Lemma}
\newtheorem{proposition}[theorem]{Proposition}
\newtheorem{problem}[theorem]{Problem}

\crefname{conjecture}{Conjecture}{Conjectures}

\crefname{question}{Question}{Questions}

\theoremstyle{remark}
\newtheorem{remark}[theorem]{Remark}

\numberwithin{equation}{section}

\usepackage{mathtools}

\DeclarePairedDelimiter\floor{\lfloor}{\rfloor}

\DeclareMathOperator{\conv}{conv}
\DeclareMathOperator{\ver}{vert}
\DeclareMathOperator{\st}{star}
\DeclareMathOperator{\lk}{link}
\DeclareMathOperator{\aff}{aff}

\newcommand{\Pl}{\mathcal{P}}

\newcommand{\C}{\mathcal{C}}
\date{\today}
\title[Almost simplicial polytopes]{Almost simplicial polytopes: the lower and upper bound theorems}

\author{Eran Nevo}
\address{Institute of Mathematics, Hebrew University of Jerusalem}
\email{\texttt{nevo@math.huji.ac.il}}
\thanks{Research of E.~Nevo was partially supported by Israel Science Foundation grants ISF-805/11 and ISF-1695/15.}
\author{Guillermo Pineda-Villavicencio}
\address{Centre for Informatics and Applied Optimisation, Federation University Australia}
\email{\texttt{work@guillermo.com.au}}
\author{Julien Ugon}
\address{School of Information Technology, Deakin University}
\email{\texttt{julien.ugon@deakin.edu.au}}
\author{David Yost}
\address{Centre for Informatics and Applied Optimisation, Federation University Australia}
\email{\texttt{d.yost@federation.edu.au}}


\begin{document}
\maketitle
\begin{abstract}
We study $n$-vertex $d$-dimensional polytopes with at most one nonsimplex facet with, say, $d+s$ vertices, called {\it almost simplicial polytopes}. We provide tight lower and upper bound theorems for these polytopes as functions of $d,n$ and $s$,
thus generalizing the classical Lower Bound Theorem by Barnette and Upper Bound Theorem by McMullen, which treat the case of $s=0$.
We characterize the minimizers and provide examples of maximizers, for any $d$.  Our construction of maximizers is a generalization of cyclic polytopes, based on a suitable variation of the moment curve, and is of independent interest.
\end{abstract}

\section{Introduction}
In 1970 McMullen \cite{McM70a} proved the Upper Bound Theorem (UBT) for {\it simplicial polytopes}, polytopes with each facet being a simplex, while between 1971 and 1973 Barnette \cite{Bar71,Bar73} proved the Lower Bound Theorem (LBT) for the same polytopes. Both results are major achievements in the combinatorial theory of polytopes; see, for example, the books \cite{Gru03, Zie95} for further details and discussion.

These results can be phrased as follows: let $C(d,n)$ (resp. $S(d,n)$) denote a cyclic (resp.~stacked) $d$-polytope on $n$ vertices, and for a polytope $P$ let $f_i(P)$ denote the number of its $i$-dimensional faces. Then the classical LBT and UBT read as follows.
\begin{theorem}[Classical LBT and UBT]
For any simplicial $d$-polytope $P$ on $n$ vertices, and any $0\le i\le d-1$,
\[f_i(S(d,n))\le f_i(P) \le f_i(C(d,n))
.\]
\end{theorem}
The numbers $f_i(S(d,n))$ and $f_i(C(d,n))$ are explicit known functions of $(d,n,i)$, to be discussed later.

We generalize the UBT and LBT to the following context:
consider a pair $(P,F)$ where $P$ is a polytope, $F$ is a facet of $P$, and all facets of $P$ different from $F$ are simplices. We call such a polytope
$P$ an \emph{almost simplicial polytope} (ASP) and a pair $(P,F)$ an ASP-pair. Since every ridge of the facet $F$ is shared with another facet of $P$, the facet $F$ is necessarily simplicial. We will be interested only in the combinatorics of $P$.

Let $\mathcal{P}(d,n,s)$ denote the family of $d$-polytopes $P$ on $n$-vertices such that $(P,F)$ is an ASP-pair, where $F$ has $d+s$ vertices ($s\ge 0$). Note that $\mathcal{P}(d,n,0)$ consists of the simplicial $d$-polytopes on $n$ vertices.
In this paper, we define certain polytopes $C(d,n,s),S(d,n,s)\in \Pl(d,n,s)$, explicitly compute their face numbers, and show the following.
\begin{theorem}[LBT and UBT for ASP]\label{thm:Main}
For any $d,n,s$, any polytope $P\in \Pl(d,n,s)$, and any $0\le i\le d-1$,
\[f_i(S(d,n,s))\le f_i(P) \le f_i(C(d,n,s))
.\]
Further, for $d\geq 4$, the polytopes $P\in \Pl(d,n,s)$ with $f_i(P)=f_i(S(d,n,s))$ for \emph{some} $1\le i\le d-1$ are characterized combinatorially, and satisfy the above equality for \emph{all} $0\le i\le d-1$.
\end{theorem}

The characterization of the equality case above generalizes Kalai's result \cite{Kal87} that for $d\geq 4$ equality in the classical LBT holds for some $1\le i\le d-1$ if and only if $P$ is stacked.
The polytopes $C(d,n,s)$ form an ASP analog of cyclic polytopes and satisfy a combinatorial Gale-evenness type description of their facets.

The combinatorics of $P$ could be also understood by looking at certain triangulations of $F$. Consider an ASP-pair $(P,F)$. Projectively transform $P$ into a combinatorially equivalent polytope such that the orthogonal projection of $\mathbb R^{d}$ onto the hyperplane spanned by the facet $F$ maps $P\setminus F$ into the relative interior of $F$; see \cite[Ex.~2.18]{Zie95}. Think of $F$ as sitting in $(\mathbb R^{d-1},0)$. Under this setting,  the facet $F$ admits a regular triangulation $\C$ which is obtained by a lifting of the vertices of $\C$  which leaves the vertices of $F$ fixed; see \cite[Sec.~17.3]{GooORo04}. Then the polytope $P$ becomes the convex hull of the lifted vertices of $\C$. Consequently, specifying the ASP-pair $(P,F)$ amounts to specifying the aforementioned triangulation of $F$.

We pay special attention to simplicial balls of the form $P':=\partial P \setminus \{F\}$, which are a subfamily of certain balls considered by Billera and Lee \cite{BilLee81a} in their study of polytope pairs. In particular, their results give tight upper and lower bound theorems for the face numbers of simplicial $(d-1)$-dimensional balls of the ``\emph{polytope-antistar}" form; that is, balls of the form $\partial Q\setminus \{v\}$, or $\partial Q\setminus v$ for simplicity, where $Q$ is a simplicial $d$-polytope and $v$ is a vertex of $Q$ that is deleted.
These bounds are given as functions of $d,f_0(\partial Q\setminus v),f_0(Q/v)$, where $Q/v$ denotes the vertex figure of $v$ in $Q$. For an  ASP-pair $(P,F)$, let $Q$ be obtained from $P$ by stacking a pyramid over $F$ with a new vertex $v$; the  stacking operation is defined in detail in \cref{sec:LBT}. Then
$F\cong Q/v$ and $P'=\partial P \setminus \{F\} = \partial Q\setminus v$. Thus, our balls $P'$ form a subfamily of the balls $\partial Q\setminus v$ considered in \cite{BilLee81a}. The bounds we obtain in Theorem~\ref{thm:Main} are strictly stronger than those of \cite{BilLee81a} which apply to all polytope-antistar balls.

Let $f(P)=(1,f_0(P),f_1(P),\cdots,f_{d-1}(P))$ denote the {\it $f$-vector} of $P$, a vector recording the face numbers of $P$.
The following problem naturally arises.
\begin{problem}
Characterize the pairs of $f$-vectors $(f(P),f(F))$ for ASP-pairs $(P,F)$.
\end{problem}
A solution to the problem above would generalize the well known $g$-theorem characterizing the face numbers of simplicial polytopes, conjectured by McMullen \cite{McMullen-g-conj} and proved by Billera-Lee \cite{Billera-Lee} (sufficiency) and Stanley \cite{Sta80} (necessity); the $g$-theorem solves the case $s=0$ and provides some restrictions when $s>0$.
We leave this general problem to a future study.
We remark that for the corresponding problem for the larger family of polytope pairs \cite{BilLee81a}, currently there is no conjectured characterization, after Kolins' \cite{Kolins11} counterexamples to the characterization conjectured by Billera and Lee \cite{BilLee81a}.

The proof of the LBT for ASP and the characterization of the equality cases are based on framework-rigidity arguments (cf.~Kalai~\cite{Kal87}) and on an adaptation of the well known McMullen-Perles-Walkup reduction (MPW reduction) \cite[Sec.~5]{Kal87} to ASP; see Section~\ref{sec:LBT}.

The numerical bounds obtained in the UBT for ASP are a special case of a recent result of Adiprasito and Sanyal~\cite[Thm.~3.9]{AdiSan14}, who proved the bounds for homology balls whose boundary is an induced subcomplex. While their proof relies on machinery from commutative algebra, our proof is elementary and is based on a suitable shelling of $P$. Further, our construction of maximizers $C(d,n,s)$ is a generalization of cyclic polytopes, based on a suitable variation of the moment curve, and is of independent interest;
see Section~\ref{sec:UBT}.

Our proof techniques are likely to extend beyond ASP to polytopes $P$ where \emph{all} non-simplex proper faces $F_{1},\ldots, F_{n}$ are facets, but probably not beyond that, as a key fact that  we use in this paper, and probably will need in the extended setting, is that the $f$-vector of $P$ can be recovered from  the $f$-vector of $P':=\partial P\setminus \{F_1,\ldots,F_n \}$ and $n$.

\section{Preliminaries}\label{sec:prelim}
For undefined terminology and notation, see \cite{Zie95} for polytopes and complexes, or \cite[Sec.~2]{Kal87} for framework rigidity.

\subsection{Polytopes and simplicial complexes}
The $k$-dimensional faces of a polyhedral complex $\Delta$ are called
$k$-faces, where the empty face has dimension $-1$.
For  a simplicial complex $\Delta$ of dimension $d-1$, the numbers $f_k(\Delta)$ are then related to the $h$-numbers
$h_k(\Delta):=\sum_{i=0}^k (-1)^{k-i}{d-i \choose k-i}f_{i-1}(\Delta)$ by
\begin{align}\label{eq:fk_hk}
f_{k-1}(\Delta)=\sum_{i=0}^k {d-i\choose k-i}h_i(\Delta).
\end{align}
The $h$-vector of $\Delta$, $(\ldots,h_k,h_{k+1},\ldots)$, can be considered as an infinite sequence if we let $h_k(\Delta)=0$ for $k>d$ and $k<0$.
The $g$-numbers are defined by $g_k(\Delta)=h_k(\Delta)-h_{k-1}(\Delta)$.

For an ASP pair $(P,F)$, where $P$ is $d$-dimensional,
the following version of the Dehn-Somerville equations applies to the complex  $P'=\partial P\setminus \{F\}$.
\begin{proposition}[{\cite[Thm.~18.3.6]{GooORo04}}, Dehn-Somerville Equations for $P'$]\label{prop:Dehm-Somerville}
The $h$-vector of the simplicial $(d-1)$-ball
$P'$ with boundary $\partial F$
satisfies for $k=0,\ldots,d$
\begin{align}\label{eq:DSEq-hP}
h_k(P')&=h_{d-k}(P')+g_{k}(\partial F).
\end{align}
\end{proposition}
Note that $h_{k}(P')=0$ and $h_k(\partial F)=0$  for $k\ge d$ and $h_{d-1}(\partial F)=1$.

We proceed with a number of definitions related to simplicial complexes. Let $2^A$ denote the simplicial complex generated by the set $A$; it is a simplex. Sometimes we abbreviate this complex by $A$, when the context is clear.
Say $\Delta$ is \emph{pure} if all its maximal faces, called \emph{facets}, have the same dimension, and a pure simplicial complex $\Delta$ is \emph{shellable} if its facets can be ordered $F_1,F_2,\ldots$ such that for each $j>1$, $2^{F_j}$ intersects the complex $\cup_{i<j}2^{F_i}$ in a pure codimension $1$ subcomplex of $2^{F_j}$. Such an order is called a \emph{shelling order} or \emph{shelling process} of $\Delta$.
For a shelling order, the set of faces $2^{F_j} \setminus \cup_{i<j}2^{F_i}$ has a unique minimal element, called the \emph{restriction face} of $F_j$, denoted $R_j$.
For \emph{any} shelling of $\Delta$, $h_i(\Delta)$ equals the number of facets in the shelling  whose restriction face has size $i$; cf. \cite[Thm.~8.19]{Zie95}. Note that $P'$ is shellable, by a Bruggesser-Mani line shelling \cite[Prop.~2]{BruMan71}.

The \emph{link} of a face $F$ in the simplicial complex $\Delta$ is $\lk_{\Delta}(F):=\{T\in \Delta:\ T\cap F =\emptyset,\ F\cup T\in \Delta\}$, and its \emph{star}, $\st_{\Delta}(F)$ is the complex $\cup_{F\subseteq T}2^T$. Thus, using the \emph{join} operator on simplicial complexes, we obtain $2^F * \lk_{\Delta}(F)=\st_{\Delta}(F)$. For a general polyhedral complex, the star of a face $F$ is the polyhedral subcomplex formed by all faces containing $F$, and their faces.
For a vertex $v$ in a polytope $Q$, its \emph{vertex figure}
$Q/v$ is a codimension $1$ polytope obtained by intersecting $Q$ with a hyperplane $H$ \emph{below} $v$, that is, $v$ is on one side of $H$ and the other vertices of $Q$ are on the other side. If $\st_Q(v)$ is simplicial then the boundary complex of $Q/v$ coincides with $\lk_{Q}(v)$.

A subcomplex $K$ of $\Delta$ is {\it induced} if it contains all the faces in $\Delta$ which only involve vertices in $K$. Note that, for an ASP-pair $(P,F)$, $\partial F$ is an induced subcomplex of $P'$, by convexity.

The underlying set $|\mathcal{C}|$ of a polyhedral complex $\mathcal{C}$ is the point set $\cup_{Q\in \mathcal{C}} Q$ of its geometric realization.
A \emph{refinement} (or subdivision) of $\mathcal{C}$ is another polyhedral complex $\mathcal{D}$ such that $|\mathcal{D}|=|\mathcal{C}|$ and for any face $F\in \mathcal{D}$ there exists a face $T\in \mathcal{C}$ such that $|F|\subseteq |T|$.

A simplicial complex $\Delta$ is a \emph{homology sphere} (over a fixed field $\mathbf{k}$) if for any face $F\in \Delta$, the reduced homology groups $\tilde{H}_i(\lk_{\Delta}(F); \mathbf{k})\cong \tilde{H}_i(S^{\dim \Delta -\dim F -1}; \mathbf{k})$ for all $i$, where $S^j$ is the $j$-dimensional sphere.
Say $\Delta$ is a \emph{homology ball} if $\tilde{H}_i(\lk_{\Delta}(F); \mathbf{k})$ vanishes for $i<\dim \Delta -\dim F -1$ and is isomorphic to either $0$ or $\mathbf{k}$ for $i=\dim \Delta -\dim F -1$. Furthermore, the boundary complex $\partial \Delta$ of $\Delta$, consisting of all faces $F$ for which $\tilde{H}_{\dim \Delta -\dim F -1}(\lk_{\Delta}(F); \mathbf{k})=0$, is a homology sphere (of codimension $1$). In particular, simplicial spheres (resp. balls) are homology spheres (resp. balls). For  the ASP-pair $(P,F)$, the complex $P':=\partial P \setminus \{F\}$ is a shellable simplicial $(d-1)$-ball and its boundary complex $\partial P'$, which coincides with $\partial F$, is a homology sphere.

\subsection{Rigidity}
We mostly follow the presentation in Kalai's paper \cite{Kal87}.
Let $G=(V,E)$ be a graph, and $\textrm{dist}(a,b)$ denote Euclidean distance
between points $a$ and $b$ in a Euclidean space.
A $d$-embedding $\alpha:V\rightarrow \mathbb{R}^{d}$ is called \emph{rigid}
if there exists an $\varepsilon>0$ such that if $\beta:V\rightarrow
\mathbb{R}^{d}$ satisfies $\textrm{dist}(\alpha(v),\beta(v))<\varepsilon $ for every
$v\in V$ and $\textrm{dist}(\beta(u),\beta(w))=\textrm{dist}(\alpha(u),\alpha(w))$ for every $\{u,w\}\in E$,
then $\textrm{dist}(\beta(u),\beta(w))=\textrm{dist}(\alpha(u),\alpha(w))$ for every $u,w\in V$. The graph $G$ is said to be \emph{generically d-rigid} if the set of its rigid
$d$-embeddings is open and dense in the metric vector space
of all of its $d$-embeddings.
Given a $d$-embedding $\alpha:V\rightarrow \mathbb{R}^{d}$, a \emph{stress}
on $\alpha$ is a function $w:E\rightarrow \mathbb{R}$ such that for
every vertex $v\in V$
\[\sum_{u:\{v,u\}\in E}w(\{v,u\})(\alpha(v)-\alpha(u)) =0.\]
The stresses on $\alpha$ form a vector space, called the \emph{stress space}.
Its dimension is the same for all generic $d$-embeddings (namely, for an open and dense set in the space of all $d$-embeddings of $G$). A graph
$G$ is called \emph{generically d-stress free} if this dimension is zero.

If a generic $\alpha:V\rightarrow \mathbb{R}^{d}$ is rigid, then
$f_1(G)\ge df_0(G)-{d+1 \choose 2}$ \cite[Sec.~3]{AsiRot78}.
Thus, if $\Delta$ is a simplicial complex of dimension $d-1$ whose $1$-skeleton is generically $d$-rigid, then
$f_1(\Delta)\ge df_0(\Delta)-{d+1 \choose 2}$, and $g_2(\Delta)$ is the dimension of the stress space of any generic embedding.
Based on these observations for $\Delta$ the boundary of a simplicial $d$-polytope with $d\ge 3$, and more general complexes, Kalai~\cite{Kal87} extended the LBT and characterized the minimizers.

For a $d$-polytope $P$ with a simplicial $2$-skeleton, the so called toric $g_2(P)$ coincides with $g_2(\partial P)=f_1(P)-df_0(P)+\binom{d+1}{2}$. By
a result of Alexandrov (cf.~Whiteley~\cite{Whi84}), the toric $g_2(P)$
 equals  the dimension of the stress space of the $1$-skeleton of $P$.

For our LBT for ASP, we will need the following very special case of Kalai's monotonicity\footnote{Kalai's monotonicity conjecture on the toric $g$-polynomials, asserting that $g(P)\ge g(F)g(P/F)$ coefficientwise for any face $F$ of $P$, was first proved for rational polytopes by Braden and MacPherson \cite{BraMac99}. Later, using the theory of combinatorial intersection homology, Braden \cite{Bra06} proved Kalai's conjecture in full generality.},
which Kalai proved using rigidity arguments.

\begin{theorem}[Kalai's Monotonicity {\cite[Thm.~4.1]{Kal94}}, weak form]\label{thm:kalaiMonotonicity}
Let $d\ge 4$, $P$ a $d$-polytope with a simplicial $2$-skeleton, and $F$ a facet of $P$. Then
\[g_2(P)\ge g_2(F).\]
Equivalently,
$f_1(P)-f_1(F)\ge (df_0(P)-\binom{d+1}{2})-((d-1)f_0(F)-\binom{d}{2})$.
\end{theorem}

\section{A lower bound theorem for almost simplicial polytopes}
\label{sec:LBT}

Before proving the lower bound theorem, we give several definitions which we use in the section.

Let $G$ be a proper face of a polytope $Q$. A point $w$ is {\it beyond} $G$ (with respect to $Q$) if
(i) $w$ is not on any hyperplane supporting a facet of $Q$, (ii)
$w$ and the interior of $Q$ lie on different sides of any hyperplane supporting a facet containing $G$, but (iii) on the same side of every other facet-defining hyperplane which does not contain $G$.
For an ASP-pair $(P,F)$ we will consider the simplicial polytope $Q$ obtained as the convex hull of $P$ and a vertex $y$ beyond $F$.

A polytope is {\it $k$-simplicial} if  each $k$-face is a simplex; a $(d-1)$-simplicial $d$-polytope is simply a simplicial $d$-polytope. A simplicial $d$-polytope is called \emph{stacked} if it can be obtained from a $d$-simplex by repeated \emph{stacking}, namely, adding a vertex beyond a facet and taking the convex hull. While stacked $d$-polytopes on $n$ vertices, denoted $S(d,n)$, may have different combinatorial structures, they all have the same $f$-vector, given by \[f_k(S(d,n))=\phi_{k}(d,n):=\begin{cases}
{d\choose k}n-{d+1\choose k+1}k&\text{for $k=1,\ldots,d-2$}\\
(d-1)n-(d+1)(d-2) &\text{for $k=d-1$}.\end{cases}\]

A homology sphere is {\it stacked} if it is combinatorially isomorphic to the boundary complex of a stacked polytope.

For any integers $d\ge 3$, $s\ge 0$ and $n\ge d+s+1$, let $F$ be a stacked $(d-1)$-polytope with $d+s$ vertices.
Construct a pyramid over $F$ and then stack $n-d-s-1$ times over facets of the resulting polytope that are different from $F$ to obtain a polytope $S(d,n,s)$ in $\Pl(d,n,s)$.
One easily computes the $f$-vector of $S(d,n,s)$, since refining $F$ by its (unique) stacked triangulation refines the boundary complex of $S(d,n,s)$ to a stacked simplicial sphere with $f$-vector $f(S(d,n))$. We obtain
\[f(S(d,n,s)) = f(S(d,n)) - (0,0,\cdots,0,s,s)
.\]

We are ready to state the LBT for ASP (\cref{thm:LBT}); its minimizers will be characterized later (see \cref{thm:LBT>4,thm:LBT=4}). In the proof of \cref{thm:LBT} we rely on the MPW reduction, which states that if the result is true for the edges of a simplial polytope then it is true for all faces of all dimensions of the polytope. This reduction is clearly explained in the proof of \cite[Thm.~1]{Bar73}. For almost simplicial polytopes, the same reasoning gives that, if the result is true for edges,  it is true for all faces of  dimensions at most $d-3$. See also \cite[Sec.~5]{Kal87}.

\begin{theorem}[LBT for ASP]\label{thm:LBT}
Let $d\ge 3$, $s\ge 0$, $n\ge d+s+1$. Then for any $P\in \Pl(d,n,s)$ and $1\le i \le d-1$ we have
\[f_i(S(d,n,s))\le f_i(P)
.\]
\end{theorem}
\begin{proof}
We proceed by induction on $d$, with the case $d=3$ as the basis. For $d=3$ and $n\ge s+4$, any $P\in \Pl(3,n,s)$ has $f$-vector $f(P)=(1,n,3n-6-s,2n-4-s)=f(S(3,n,s))$. Let $d\ge 4$.

As $P$ is $2$-simplicial,
by a result of Whiteley~\cite[Thm. 8.6]{Whi84}, the $1$-skeleton of $P$ is generically $d$-rigid, hence $f_1(P)\ge \phi_{1}(d,n)$, and by the MPW reduction (as usual, counting pairs $(v,A)$ such that $v$ is a vertex in an $i$-face $A$), $f_i(P)\ge \phi_{i}(d,n)$ for all $2\le i\le d-3$ as well; see \cite[Thm.~12.2]{Kal87}
\footnote{Kalai's theorem contains a typo.  It includes the case $i=k$, while it holds only for $i<k$, where $P$ is $k$-simplicial. Our ASP $P$ is $(d-2)$-simplicial.}. The rest of the proof will deal with the cases $i=d-2,d-1$.

Denote by $(P,F)$ the ASP-pair, and by $\deg_P(v)$ the degree of a vertex $v$ in the $1$-skeleton of $P$. We now prove the inequality for the facets, by a variation of the MPW reduction.
Note that the vertex figure  $P/v$ in $P$ of any vertex $v\in \ver F$ is an ASP (with $\deg_P(v)$ vertices), while for any vertex $v\in \ver P\setminus \ver F$, $P/v$ is a simplicial polytope; cf. \cite[Thm.~11.5]{Bro83}. Furthermore, for a vertex $v\in \ver F$, letting $s_v:=\deg_F(v)-(d-1)\ge 0$ gives
$P/v \in \Pl(d-1,\deg_P(v),s_v)$.

Double counting the number of pairs $(v,A)$ for a vertex $v$ in a facet $A$ of $P$, we obtain the following inequalities:
\begin{multline*}
d(f_{d-1}(P)-1) + (d+s) = \sum_{v\in \ver P}f_{d-2}(\lk_P(v))\\
\ge
\sum_{v\in \ver P \setminus \ver F}((d-2)\deg_P(v) - d(d-3))
+ \sum_{v\in \ver F}((d-2)\deg_P(v) - d(d-3)-s_v)\\
= 2(d-2)f_1(P) - d(d-3)f_0(P) - 2f_1(F) + (d-1)(d+s)\\
\ge
2(d-2)\left[d f_0(P)-\binom{d+1}{2}\right] - d(d-3)f_0(P) - 2\left[(d-1)f_0(F)-\binom{d}{2}\right] +(d-1)(d+s)\\
= d(d-1)f_0(P) - d(d+1)(d-2) - s(d-1),
\end{multline*}
where the first inequality is by the induction hypothesis and the second inequality is by
Kalai's monotonicity Theorem~\ref{thm:kalaiMonotonicity}
 and the LBT inequality for $f_1(P)$.
Comparing the LHS with the RHS gives
\[f_{d-1}(P)\ge \phi_{d-1}(d,n)-s.\]
The inequality for $f_{d-2}(P)$ follows from the inequality for $f_{d-1}(P)$ by double counting.
Since any ridge  in $P$ is contained in exactly two facets, counting the number of pairs $(R,A)$ for a ridge $R$ in a facet $A$ of $P$,  we obtain that
\[2f_{d-2}(P)=d(f_{d-1}(P)-1)+f_{d-2}(F).\]
Applying the classical LBT to the simplicial $(d-1)$-polytope $F$ with $f_0(F)=d+s$, we get
\[2f_{d-2}(P)\ge d(f_{d-1}(P)-1)+ (d-2)(d+s) - d(d-3),\]
and applying the lower bound for $f_{d-1}(P)$ yields, after dividing both sides by $2$,
the desired lower bound $f_{d-2}(P)\ge \phi_{d-2}(d,n)-s$.
\end{proof}

We now turn our attention to characterizing the minimizers of Theorem~\ref{thm:LBT}. We start with some terminology and background.

A proper subset $A$ of the vertices of a $d$-polytope $P$ is called a \emph{missing $k$-face} of $P$ if the cardinality of $A$ is $k+1$, the simplex on $A$ is not a face of $P$, but for any proper subset $B$ of $A$ the simplex on $B$ is a face of $P$. If $A$ is a missing $(d-1)$-face of $P$ (a.k.a. \emph{missing facet}) then adding the simplex $A$ cuts $P$ into two $d$-polytopes $P_1, P_2$, glued along the simplex $A$. We denote this operation by $P=P_1\# P_2$. Repeating this procedure on each $P_i$ until no piece $P_i$ contains a missing $(d-1)$-face results in a decomposition
$P=P_1\# P_2\#\cdots\# P_t$,
where intersections along missing $(d-1)$-faces of $P$ define the edges of a tree whose vertices are the $P_i$'s.
Call such a decomposition the \emph{prime decomposition} of $P$, and call each $P_i$ a \emph{prime factor} of $P$. See \cite[Sec.~3.8]{Kal94}. For $d\ge 3$ a prime decomposition of $P$ as above is uniquely defined; this statement follows from the following simple observation.

\begin{lemma} \label{lem:uniqueness-decomposition}
Let $P$ be a $d$-polytope, $d\geq 3$. 
%In a prime decomposition $P=P_1\# P_2\#\cdots\# P_t$ of $P$, 
Then the intersection of any two missing facets of $P$ contains no interior point of $P$.
\end{lemma}
\begin{proof}
Let $A$ and $B$ be two missing facets of a $P$ and assume by contradiction that $v$ is an interior point in the intersection of $A$ and $B$. Denote by $H_A$ (resp. $H_B$) the unique hyperplane containing $A$ (resp. $B$). As $d\ge 3$ the intersection of $H_A$ and $H_B$ contains a line $\ell$ through $v$. Consider the intersection $\ell\cap A\cap B$ and denote this segment by $[u,w]$. Then $v$ is interior to $[u,w]$. If $u$ is not on the boundary $\partial A$ of $A$, it must be on $\partial B$; but the first condition says that $u$ is interior to $P$ while the second condition says that $u$ is in $\partial P$, a contradiction. Thus, both $u$ and $w$ are in $\partial A \cap\partial B$.

Let $F_u$ be the minimal face of $\partial P$ containing $u$, and define $F_w$ similarly. Both $F_u$ and $F_w$ are proper faces of both $A$ and $B$. As $v$ is interior to $A$, $\ver F_u \cup \ver F_w = \ver A$. But then $\ver A\subseteq \ver B$, implying that $A=B$, a final contradiction.
\end{proof}

By  virtue of \cref{lem:uniqueness-decomposition}, for $d\ge 3$, we denote by $\Delta_P$ the polyhedral complex defined by {\it the} prime decomposition of a $d$-polytope $P$. Then a simplicial $d$-polytope $P$ is stacked if and only if all its prime factors are $d$-simplices. This definition immediately extends to polyhedral spheres where the operation $\#$ corresponds to the topological connected sum.

We start with the characterization of the minimizers for the easier case $d>4$.

\begin{theorem}[Characterization of minimizers for $d>4$]\label{thm:LBT>4}
Let $d>4$ and $P\in \Pl(d,n,s)$. Let $\Delta_F$ be the polyhedral complex corresponding to the prime decomposition of $F$, and let $\Delta$ be the refinement of the boundary complex $\partial P$ of $P$ obtained by refining $F$ by $\Delta_F$. Assume there is some $1\le i\le d-1$ for which $f_i(P)=f_i(S(d,n,s))$.
Then, all prime factors in the prime decomposition of $\Delta$ are $d$-simplices.
In particular, $f(P)=f(S(d,n,s))$.
\end{theorem}

\begin{remark}\label{rmk:H-stacking}
Let $Q$ be a polytope, $G$ a facet of $Q$ and $H$ the hyperplane containing $G$. An \emph{$H$-stacking} on $Q$ is the operation of (i) adding a new vertex $w$ in $H$, beyond a facet of $G$ (with respect to $G$)
such that perturbing $w$ from $H$ to the side of the interior of $Q$ makes $w$ beyond a facet of $Q$, and (ii) taking the convex hull of $w$ and $Q$.
The minimizers considered in Theorem \ref{thm:LBT>4} are precisely the polytopes that can be obtained by the following recursive procedure: start with a $d$-simplex having a facet in a hyperplane $H$, and repeatedly either $H$-stack or (usual) stack over a facet not in $H$.

Clearly this procedure produces a prime decomposition of $\Delta$ of Theorem \ref{thm:LBT>4}. Conversely, consider the rooted tree corresponding to the prime decomposition of $\Delta$. Starting with the root, add vertices one by one so that the resulting induced forest is always a tree. Such ordering induces a recursive procedure as above.
\end{remark}

\begin{proof}[Proof of Theorem~\ref{thm:LBT>4}]
By the classical MPW reduction for $1\leq i\leq d-3$ and the variation of it we used in the proof of Theorem~\ref{thm:LBT} for $d-2\leq i\leq d-1$, equality for some $1\leq i\leq d-1$ implies equality for $i=1$, so
it is enough to consider the case $i=1$.
From Kalai's monotonicity (Theorem~\ref{thm:kalaiMonotonicity}) and our assumption $g_2(P)=0$, it follows that $g_2(F)=0$.
As $F$ is simplicial of dimension $\ge 4$, Kalai's \cite[Thm.~1.1(ii)]{Kal87}
says that $F$ is stacked, thus $\Delta$ is a simplicial $(d-1)$-sphere. Since $g_2(\Delta)=0$, by \cite[Thm.~1.1(ii)]{Kal87} again, $\Delta$ is stacked, as desired.

In particular, as $f(P)=f(\Delta)-(0,\ldots,0,s,s)=f(S(d,n))-(0,\ldots,0,s,s)$ we conclude that
$f(P)=f(S(d,n,s))$.
\end{proof}

For $d=4$, $F$ need not be stacked. For example, the pyramid over any simplicial $3$-polytope is a minimizer. We obtain the following characterization of minimizers.

\begin{theorem}[Characterization of minimizers for $d=4$]\label{thm:LBT=4}
Let $P\in \Pl(4,n,s)$,
and keep the notation of Theorem~\ref{thm:LBT>4}. Assume there is some $1\le i\le d-1$ for which $f_i(P)=f_i(S(d,n,s))$.
Then, the prime factors in the prime decomposition of $\Delta$ are either $d$-simplices with no facet contained in $|F|$, or pyramids over prime factors of $F$.
\end{theorem}

In order to prove this theorem we first need to show generic $d$-rigidity for the $1$-skeleton of a much larger class of complexes.

Let $\C_{k}$ be the family of homology $k$-balls $\Delta$ such that:
\begin{itemize}
\item the induced subcomplex $\Delta[I]$ on the set $I$ of internal vertices has a connected $1$-skeleton, and
\item for any edge $e$ in the boundary complex $\partial\Delta$, there exists a $2$-simplex $T$, $e\subset T$, such that $T$ has a vertex in $I$.
\end{itemize}
%%%%%%%%%%%%%
Note that any homology $k$-ball $\Delta$ whose boundary $\partial\Delta$ is an induced subcomplex is in $\C_k$\footnote{This follows from Alexander duality, however we preferred to give an easier argument.}.
Indeed, as $\partial\Delta$ is induced, any facet of $\Delta$ intersects $I$ nontrivially. Assume by contradiction that the graph $\Delta[I]_{\leq 1}$ is disconnected, say equals the disjoint union of nontrivial graphs $G_1$ and $G_2$. As $\Delta$ is facet-connected it has facets $F_1$ and $F_2$ whose intersection $S$ has codimension $1$ and $F_i$ has a vertex in $G_i$ for $i=1,2$. Then $S$ is disjoint from $I$, so $S\in \partial\Delta$. This is a contradiction as $S$ is contained in two facets of $\Delta$, not in one.

In particular,
for $P\in \Pl(d,n,s)$,
the simplicial complex $P'=\partial P\setminus\{F\}$ is in $\C_{d-1}$.

\begin{lemma}\label{lem:rigidity_C_d}
Let $d\ge 4$.  The $1$-skeleton of any $\Delta\in \C_{d-1}$ is generically $d$-rigid. Thus, $f_1(\Delta)\ge df_0(\Delta) - \binom{d+1}{2}$.
\end{lemma}

\begin{proof}
The proof follows from Kalai's proof of the classical LBT. Apply \cite[Prop. 6.4]{Kal87} with the tree $T$ there be a spanning tree of the $1$-skeleton of $\Delta[I]$.
\end{proof}

\begin{proof}[Proof of Theorem~\ref{thm:LBT=4}]
Consider a prime factor $L$ of $\Delta$.
Then $L$ is a $4$-polytope
which is $2$-simplicial so it has a generically $4$-rigid $1$-skeleton by \cite[Thm. 8.6]{Whi84}.
As $g_2(P)=0$, the $1$-skeleton of $L$, denoted by $G$, must be generically $4$-stress free. Thus, $g_2(L)=0$.

If $L$ does not contain a facet in $\Delta_F$, then $L$ is simplicial, with $g_2(L)=0$, hence is stacked by \cite[Thm.~1.1]{Kal87}. Being also prime, $L$ is a $4$-simplex.

Assume then that $L$ contains a facet $F''$ contained in $|F|$, so $(L,F'')$ is an ASP-pair. If $L$ has a unique vertex outside $F''$, then $L$ is a pyramid over a prime factor of $F$ and we are done. Assume the contrary, so there is an edge $vu\in G$ with $v,u\notin F''$ (for concreteness, taking $v,u$ to be the highest two vertices of $L$ above the hyperplane of $F$ works).

First we show that $vu$ satisfies the link condition $\lk_L(v)\cap \lk_L(u) = \lk_L(vu)$, which guarantees that contracting the edge $vu$ in the simplicial complex $\partial L\setminus\{F''\}$ results in $\tilde{\Delta}\in \C_3$; see e.g.\cite[Prop.2.4]{Nevo-Novinsky}\footnote{To apply \cite[Prop.~2.4]{Nevo-Novinsky}, phrased for homology spheres, simply cone the boundary of the homology ball $\Delta$ to obtain a homology sphere.}.
Indeed, if $vu$ fails the link condition it means that $vu$ is contained in a missing face $M$, with $3$ or $4$ vertices. Now, $M$ cannot have $4$ vertices as $L$ is prime. If $M=vuz$ then $uz$ is an edge of $L$ \emph{not} in $\lk_L(v)$. Since $\lk_L(v)$ is a homology $2$-sphere (thus, a simplicial $2$-sphere), its $1$-skeleton is generically $3$-rigid. Consequently, the $1$-skeleton of $\st_L(v)$ is
generically $4$-rigid, and adding $uz$ to it yields a $4$-stress in $G$, a contradiction.

Let $m$ be the number of vertices in the cycle $\lk_L(vu)$, then
$f_1(\tilde{\Delta})=f_1(L)-m-1$ and $f_0(\tilde{\Delta})=f_0(L)-1$, which implies that $g_2(L)=g_2(\tilde{\Delta})+(m-3)$.

If $m>3$, then applying Lemma~\ref{lem:rigidity_C_d} to $\tilde{\Delta}$ yields $g_2(L)>0$, a contradiction. So assume $m=3$.

Denote by $x,y,z$ the vertices of $\lk_L(vu)$.
If the triangle $xyz\in L$, then, as $L$ is prime, both tetrahedra $xyzv, xyzu$ are faces of $L$, so $L$ is the $4$-simplex $xyzuv$, a contradiction (as it has a facet $F''$ in $F$).

We are left to consider the case $xyz\notin L$.
The argument here is inspired by Barnette~\cite[Thm.~2]{Bar71}.
In this case, the $3$-ball formed by the join $vu * \partial(xyz)$ is an induced subcomplex of $\partial L\setminus\{F''\}$. Now replace it by $\partial(vu) * xyz$ (this is a bistellar move) to obtain from $\partial L\setminus\{F''\}$ the complex $\Delta"$.
Clearly $\Delta"$ is a homology $3$-ball, and any edge on its boundary is part of a $2$-simplex with an internal vertex (just take the same one as in $\partial L\setminus\{F''\}$). To show $\Delta"\in\C_3$ we are left to show that the graph on the internal vertices $I$ of $\Delta"$ is connected.
Assume not, namely removing the edge $uv$ disconnects the induced graph on $I$ in $\partial L\setminus\{F''\}$. In particular, $x,y,z\in F''$. But $xyz\notin L$, so $xyz$ is a missing face of $F''$, contradicting that $F''$ is a prime factor of $F$.

We conclude that $\Delta"\in C_3$, thus, by Lemma~\ref{lem:rigidity_C_d}, $\Delta"\cup\{vu\}$ has a nonzero $4$-stress. However, the $1$-skeletons of $\Delta"\cup\{vu\}$ and of $L$ are equal graphs
 so $g_2(L)>0$, a contradiction.
The proof is then complete.
\end{proof}

\section{An upper bound theorem for almost simplicial polytopes}
\label{sec:UBT}

Throughout this section, we let $P\in \Pl(d,n,s)$ denote an almost simplicial polytope, $(P,F)$ the ASP-pair, and $P'=\partial P \setminus \{F\}$ the corresponding shellable simplicial $(d-1)$-ball. Recall that $\partial P'=\partial F$ is an induced subcomplex of $P'$.

\subsection{ASP generalization of cyclic polytopes}
\label{subsec:almostCycPol}

The {\it moment curve} in $\mathbb{R}^d$ is defined by $t\mapsto(t,t^2,\ldots,t^d)$ for $t\in \mathbb{R}^d$, and the convex hull of any $n$ points on it gives, combinatorially, the cyclic polytope $C(d,n)$; see, for instance, \cite[Example 0.6]{Zie95}. We extend this construction by considering curves $x(t)$ of the form $(t,t^2,\ldots,t^{d-r},p_1(t),\ldots,p_r(t))$, where $p_i(t)$ is a continuous function in $t$ for $i=1,\ldots,r$.
Later, a special choice of the curve $x(t)$ and points on it will give, by taking the convex hull, our maximizer polytope $C(d,n,s)$.

We let $V(t_1,\ldots,t_{l})$ denote the {\it Vandermonde determinant} on variables $t_1,\ldots,t_l$.
\[V(t_1,\ldots,t_{l}):=\begin{vmatrix}
1 &1&\cdots &1\\
t_1&t_2&\cdots &t_{l}\\
t_1^2&t_2^2&\cdots &t_{d}^l\\
\vdots&\vdots&\cdots&\vdots\\
t_1^{l-1}&t_2^{l-1}&\cdots &t_{l}^{l-1}
\end{vmatrix}
= \prod_{1\le i<j\le l}(t_j-t_i).
\]

Recall a  polytope is {\it $k$-neighborly} if each subset of at most $k$ vertices forms the vertex set of a face. A $\floor{d/2}$-neighborly $d$-polytope is simply called  {\it neighborly}.

\begin{lemma}\label{lem:neighborly}
Consider the curve $x(t)$. Then the following holds:
\begin{enumerate}
\item Any $d-r+1$ points on the curve $x(t)$ are affinely independent.
\item For any $n$ distinct numbers $t_1,\ldots,t_n$, the polytope $Q=\conv (\{x(t_1),\ldots,x(t_{n})\})$ is $(d-r-1)$-simplicial.
\item The polytope $Q$ is $\floor{(d-r)/2}$-neighborly.
\end{enumerate}
\end{lemma}
\begin{proof}Consider  $n$ real numbers  $t_1<\ldots<t_{n}$ and the corresponding points $x(t_i)$. From any $d-r+1$ points $x(t_{i_1}),\ldots, x(t_{i_{d-r+1}})$,
we form a matrix by considering them as its columns, in this same order, and adding a row of ones as the first row. The top
$(d-r+1)\times (d-r+1)$
square submatrix of this martix has
determinant $V(t_{i_1},\ldots,t_{i_{d-r+1}})$ which is nonzero, which in turn implies the first assertion.
The second assertions follows immediately from the first.

To prove the third assertion proceed as in \cite[Sec.~4.7]{Gru03}. Consider a set $S_k=\{x(t_{i_j}):j=1,\ldots,k\}$, $1\le i_j\le n$, with $k\le \floor{(d-r)/2}$, and the polynomial
\[\beta(t)=\prod_{i=1}^k(t-t_{i_j})^2=\beta_0+\beta_1t+\cdots+\beta_{2k}t^{2k}.\]
Let $b=(\beta_1,\ldots,\beta_{2k},0,\ldots,0)$ be a vector in $\mathbb{R}^d$ and $H=\{x\in \mathbb{R}^d: x\cdot b=-\beta_0\}$ a hyperplane in $\mathbb{R}^d$. Here $\cdot$ denotes the dot product of vectors.

All the points in $S_k$ are clearly contained in $H$, and for any other $x(t_{l})\in \{x(t_1),\ldots,x(t_n)\}\setminus S_k$ we have $x(t_l)\cdot b=-\beta_0+\beta(t_l)>-\beta_0$.
Thus, $S_k$ is the vertex set of a simplex face of $Q$.
\end{proof}
	
Let $n$ and $s$ be fixed integers with $n>d+s$ and $s\ge 0$ and
consider the curve $y(t)=(t,t^2,\ldots,t^{d-1},p(t))$, where \[p(t):=(n-1)^{(t-1)(d-1)}t(t+1)\cdots(t+d+s-1).\]
Let $t_i=-s-d+i$ for $i=1,\ldots,n$. The polynomial $p(t)$ has been chosen so that $p(t_{i})=0$ for $i\in\{1,\ldots,d+s\}$ and $p(t_{i})>0$ otherwise. Let $C(d,n,s):=\conv (\{y(t_1),\ldots,y(t_{n})\})$.  Also, let $T=\{t_{i}:i=1,\ldots,n\}$, $I=\{t_{i}:i=1,\ldots,d+s\}$ and $y(S):=\{y(t_{i}):t_i\in S\}$ for $S\subset T$.

The following proposition collects a number of properties of the $d$-polytope $C(d,n,s)$.
	
\begin{proposition}\label{prop:C(d,n,s)} The $d$-polytope $C(d,n,s)$ ($n>d+s$) satisfies the following properties.
\begin{enumerate}
\item $C(d,n,s)\in \Pl(d,n,s)$.
\item {\bf Gale's evenness condition}: A $d$-subset $S_d$ of $\ver C(d,n,s)$ such that $S_d\not \subset I$ forms a simplex facet if and only if, for any two elements $u,v\in T\setminus S_d$, the number of elements of $S_d$ between $u$ and $v$ on the curve $y(t)$ is even.
\end{enumerate}
	 \end{proposition}

\begin{proof}
(1)
We first show that the first $d+s$ vertices span a facet $F$. Let   $z=(z_1,\ldots,z_d)\in \mathbb{R}^d$ and let
\[D((t_1,t_2,\ldots,t_d);z):=\begin{vmatrix}
1 &1&\cdots &1&1\\
t_1&t_2&\cdots &t_{d}&z_1\\
t_1^2&t_2^2&\cdots &t_{d}^2&z_2\\
\vdots&\vdots&\cdots&\vdots&\vdots\\
t_1^{d-1}&t_2^{d-1}&\cdots &t_{d}^{d-1}&z_{d-1}\\
p(t_1)&p(t_2)&\cdots &p(t_{d})&z_d\\
\end{vmatrix}.\]

Let $D(z):=D((t_1,t_2,\ldots,t_d);z)$ and consider the hyperplane $H_D:=\{z\in \mathbb{R}^d: D(z)=0\}$. The points $y(t_i)$ ($i=1,\ldots,d+s$) are all contained in $H_D$,
since the last row of $D(z)$ vanishes at all these  points $y(t_i)$; recall that $p(t_i)=0$ for $i=1,\ldots,d+s$. Also, by \cref{lem:neighborly}, any $d$ of these points $y(t_i)$ ($i=1,\ldots,d+s$) are affinely independent. So in fact $H_D$ equals the affine span of the points $y(t_i)$ ($i=1,\ldots,d+s$).
Let $y(t^*)\in \ver C(d,n,s)\setminus y(I)$, then $D(y(t^*))=p(t^*)V(t_1,\ldots,t_d)>0$ since $p(t^*)>0$ and $V(t_1,\ldots,t_d)>0$.
Thus, $F$ is a facet of $C(d,n,s)$.

We now show that every other facet is a simplex.
Consider any $(d+1)$-set $\{t_{i_1}<\ldots <t_{i_d}<t_{i_{d+1}}=t^*\}\subset T$ not contained in $I$. Thus, $t^*\in T\setminus I$.
Consider the determinant $E(z):=D((t_{i_1},t_{i_2},\ldots,t_{i_d});z)$. The hyperplane $H_E:=\{z\in \mathbb{R}^d: E(z)=0\}$ contains all the points $y(t_{i_j})$ ($j=1,\ldots,d$).
We need to show that $E(y(t^*))\ne 0$.

Note that $p(t)=0$ for $t\in I$ and $p(t)>0$ for $t\in T\setminus I$. Also, note that $|t_a-t_b|\le n-1$ for $t_a,t_b\in [-s-d+1,-s-d+n]$. For the sake of clarity assume $d$ is odd; the case of even $d$  is analogous. Computing $E(y(t^{*}))$ by expanding with respect to the last row gives
\begin{align*}\left(p(t^{*})V(t_{i_1},\ldots,t_{i_d})-p(t_{i_d})V(t_{i_1},\ldots,t_{i_{d-1}},t^*)\right)+\cdots\\
\quad +\left(p(t_{i_2})V(t_{i_1},t_{i_3}\ldots,t^{*})-p(t_{i_1})V(t_{i_2},\ldots,t^*)\right).
\end{align*}
The definition of $p(t)$ implies that each pair-summand is
nonnegative and the first pair-summand is
positive, and so the determinant is positive.
Indeed, for $j>1$, if $p(t_{i_j})=0$ then also $p(t_{i_{j-1}})=0$ and the corresponding pair-summand vanishes. Otherwise, let
$V(j):=V(t_{i_1},\ldots t_{i_{j-1}},t_{i_{j+1}},\ldots,t_{i_{d+1}})$ for short.
From the definition of the values of $t_{i}$ for $i=1,\ldots,n$, it follows that $t_{i_{a}}\ge t_{i_{b}}+1$ whenever $a>b$, and consequently, that
\begin{align}\label{eq:p(t)}
p(t_{i_{j}})&=(n-1)^{(d-1)(t_{i_j}-1)}\prod_{\ell=0}^{d+s-1}(t_{i_{j}}+\ell)&> (n-1)^{(d-1)
(t_{i_{j-1}}-1)}(n-1)^{d-1}\prod_{\ell=0}^{d+s-1}(t_{i_{j-1}}+\ell)\\
&&=(n-1)^{d-1}p(t_{i_{j-1}})\nonumber.
\end{align}
And from the definition of $V(j)$ it follows that
\begin{align}\label{eq:V(j)-A}
\frac{V(j)}{V(j-1)}&=\frac{t_{i_{j+1}}-t_{i_{j-1}}}{t_{i_{j+1}}-t_{i_{j}}}\cdots\frac{t_{i_{d+1}}-t_{i_{j-1}}}{t_{i_{d+1}}-t_{i_{j}}}\frac{t_{i_{j-1}}-t_{i_{j-2}}}{t_{i_{j}}-t_{i_{j-2}}}\cdots\frac{t_{i_{j-1}}-t_{i_{1}}}{t_{i_{j}}-t_{i_{1}}}.
\end{align}
Since $1\le t_{i_{a}}- t_{i_{b}}\le n-1$ whenever $a>b$, we get
\begin{align}\label{eq:V(j)-B}\frac{t_{i_{\ell}}-t_{i_{j-1}}}{t_{i_{\ell}}-t_{i_{j}}}&\ge \frac{1}{n-1}\;\text{for $\ell=j+1,\ldots,d+1$, and}\\
\frac{t_{i_{j-1}}-t_{i_{\ell}}}{t_{i_{j}}-t_{i_{\ell}}}&\ge \frac{1}{n-1}\;\text{for $\ell=1,\ldots,j-2$,}\nonumber\end{align} for each of the $d-1$ quotients in \cref{eq:V(j)-A}. In consequence, combining \cref{eq:p(t),eq:V(j)-A,eq:V(j)-B} we finally get that
\begin{align*}
\frac{V(j)}{V(j-1)}p(t_{i_{j}})> \frac{1}{(n-1)^{d-1}}(n-1)^{d-1}p(t_{i_{j-1}}).
\end{align*}
or equivalently that
\begin{align*}
V(j)p(t_{i_{j}})> V(j-1)p(t_{i_{j-1}}),
\end{align*}
as desired.
This completes the proof of the first assertion.

(2) Consider a set $S_d=\{t_{i_1}<\ldots <t_{i_d}\}\not \subset I$. Let $t^*\in T$, $t_{i_{j-1}}<t^*<t_{i_j}$ (include also the cases $t^*<t_{i_1}$ with $j=1$ and $t_{i_d}<t^*$ where we put $j=d+1$). From the above reasoning we see that if the column $y(t^*)$ in the determinant $E(y(t^*))$ is placed between the columns $y(t_{i_{j-1}})$ and $y(t_{i_j}) $ then the resulting determinant is positive. To achieve this, we swap $d-j+1$ times the column $y(t^*)$, which gives that the sign of $E(y(t^*))$ is $(-1)^{d-j+1}$. Consequently, on the curve $y(t)$, between $[-s-d+1,-s-d+n]$, the determinant $E(y(t^*))$ changes sign whenever the variable passes through one of the values $t_{i_{j}}$ ($i=1,\ldots,d$), and we are done.
\end{proof}	

A polytope $C(d,n,s)$ will be called {\it almost cyclic}. Having established in \cref{lem:neighborly} that $C(d,n,s)$ is $\floor{(d-1)/2}$-neighborly, we can compute its $h$-vector, in steps.
Recall that $P'=\partial P \setminus \{F\}$.

\begin{proposition}\label{prop:neighbourlyASP} Let $P\in \Pl(d,n,s)$ be  $\floor{(d-1)/2}$-neighborly, and $(P,F)$ the ASP-pair. Then,
\begin{align*}
h_{k}(P')&= {n-d-1+k\choose k},&\text{if $0\le k\le \floor{(d-1)/2}$};\\
h_{d-k}(P')&=  {n-d-1+k\choose k}-{s+k-1\choose k},&\text{if $1\le k\le \floor{(d-1)/2}$}.
\end{align*}

\end{proposition}
\begin{proof} First note that $f_{k-1}(P')= {n\choose k}$ for $k\le \floor{(d-1)/2}$. Thus, it follows that
\[h_k(P')=\sum_{i=0}^k(-1)^{k-i}{d-i\choose k-i}{n\choose i}={n-d-1+k\choose k}.\]

We now consider the remaining values of $k$.
Using that $F$ is a neighborly simplicial $(d-1)$-polytope we obtain that
$g_{k}(F)= {d+s-(d-1)+k-2\choose k}$, for $0\le k\le \floor{(d-1)/2}$. Thus, from \cref{eq:DSEq-hP} of Proposition~\ref{prop:Dehm-Somerville} it follows, for $1\le k\le \floor{(d-1)/2}$, that
\[h_{d-k}(P')= {n-d-1+k\choose k}-{s+k-1\choose k}.\]\end{proof}

Observe that, for even $d$, being $\floor{(d-1)/2}$-neighborly does not determine the value of $h_{d/2}(P')$. With the help of Gale's evenness condition we can compute the number of facets of $C(d,n,s)$, and together with \cref{prop:neighbourlyASP} and \cref{eq:fk_hk}, we can compute $h_{d/2}(C(d,n,s))$ for any even $d$ as well.

\begin{proposition}\label{prop:facetsC(d,n,s)}
For the ASP-pair $(C(d,n,s),F)$ with $d$ even, consider
the simplicial ball $C':=C(d,n,s)\setminus\{F\}$. Then
\begin{align*}
f_{d-1}(C')=\left(\binom{n-d/2 -1}{d/2} + \sum_{i=0}^{d/2 -1} 2 \binom{n-d-1+i}{i}\right) - \binom{s+ d/2}{d/2}.\end{align*}
\end{proposition}
	
\begin{proof}
The counting argument for the facets of $C'$, based on Gale evenness, goes as in the proof of the number of facets of cyclic polytopes (cf.~\cite[Cor.~8.28]{Zie95}), with the difference that we discard the Gale $d$-tuples formed solely by the first $d+s$ vertices, thus we discard exactly $\binom{s+ d/2}{d/2}$ of them.
\end{proof}

\begin{corollary}\label{thm:(k-1)C(d,n,s)}
The $h$-numbers of $C'$ are given by
\begin{align*}
h_{k}(C')&={n-d-1+k\choose k},&\text{if $0\le k\le \floor{(d-1)/2}$};\\
h_{d-k}(C')&=  {n-d-1+k\choose k}-{s+k-1\choose k},&\text{if $1\le k\le \floor{d/2}$}.
\end{align*}
\end{corollary}
\begin{proof} The case of odd $d$ was already established by \cref{prop:neighbourlyASP} since $C(d,n,s)$ is $\floor{(d-1)/2}$-neighborly.
For the case of even $d$ it remains to compute $h_{d/2}(P)$.
Equating the corresponding expression in Proposition \ref{prop:facetsC(d,n,s)} with the expression of $f_{d-1}$ in \cref{eq:fk_hk}, after substituting the known values of $h_k$ for $k\neq d/2$,  gives
\begin{align*}h_{d/2}(C')&={n-d/2-1\choose d/2} + \sum_{i=0}^{d/2 -1}{s+i-1\choose i} -{s+d/2\choose d/2}\\
\quad &= {n-d/2-1\choose d/2} -{s+d/2-1\choose d/2}
,
\end{align*}
as desired.
\end{proof}
\begin{comment}
\section{Theorem 3.10 from AS's paper}

Let $h_{k}(\Delta,\partial \Delta):=h_{k}(\Delta)+h_{k-1}(\partial \Delta)-h_{k}(\partial \Delta)$, where $\Delta$ is a simplicial $d$-ball and  $\partial \Delta$ is the boundary. For an ASP $d$-polytope $P$, we have $P=\Delta$, $f_0(P)=n$, $F=\partial \Delta$, $f_0(F)=d+s$ and $h_{d-k}(\Delta)=h_{k}(\Delta,\partial \Delta)$. Thus, the inequalities of Theorem 3.10 can be written as:

\begin{align*}
h_{k}(\Delta)&\le {n+d-k-2\choose d+1-d+k}={n+d-k-2\choose 1+k},&\text{if $0\le k\le \floor{(d-1)/2}$}\\
h_{d-k}(\Delta)&\le {n-d+k-2\choose k} -{s+k-2\choose k},&\text{if $0\le k\le \floor{(d-1)/2}$}.
\end{align*}
\end{comment}

\subsection{An upper bound theorem for almost simplicial polytopes}
We are now in a position to state an upper bound theorem for almost simplicial polytopes $P\in \Pl(d,n,s)$.

\begin{theorem}[UBT for ASP]\label{thm:UBT_ASP}
Any almost simplicial polytope $P\in \Pl(d,n,s)$ satisfies
\begin{align}
h_{k}(P')&\le {n-d-1+k\choose k},&\text{if $0\le k\le \floor{(d-1)/2}$};
\label{eq:hk}\\
h_{d-k}(P')&\le  {n-d-1+k\choose k}-{s+k-1\choose k},&\text{if $1\le k\le \floor{d/2}$}.
\label{eq:hd-k}
\end{align}
Thus,
\[f_{i-1}(P)\le f_{i-1}(C(d,n,s)\quad  \text{for $i=1,2,\ldots,d$},\]
for the almost cyclic $d$-polytope $C(d,n,s)$. Equality for some $f_{i-1}$ with $\floor{(d-1)/2}\le i\le d$ implies that $P$ is $\floor{(d-1)/2}$-neighborly.
\end{theorem}
\begin{proof}[Proof of \cref{thm:UBT_ASP} via  {\cite[Thm.~3.9]{AdiSan14}}]
The inequalities on $h_k(P')$ hold for $0\le k\le d-1$ by \cite[Thm.~3.9]{AdiSan14}, as $P'$ is a special case of a homology ball whose boundary is an induced subcomplex.
From Corollary~\ref{thm:(k-1)C(d,n,s)} and \cref{eq:fk_hk} the inequality $f_{i-1}(P)\le f_{i-1}(C(d,n,s)$ follows. Equality for some $f_{i-1}$ with  $d\ge i\ge \floor{(d-1)/2}$ implies, by \cref{eq:fk_hk}, the equality $h_{k}(P')={n-d-1+k\choose k}$ for all $0\le k\le \floor{(d-1)/2}$, and thus, again by \cref{eq:fk_hk},
that $P$ is $\floor{(d-1)/2}$-neighborly.
\end{proof}

\begin{remark}[More maximizers.]\label{rem:MoreMaximizers}
As is the case with neighborly polytopes, we expect that there are many combinatorially distinct ASPs achieving the upper bounds in the UBT for ASP. Here we sketch another such construction, based on a certain perturbation of the Cayley polytope constructed by Karavelas and Tzanaki \cite[Sec.5]{KarTza12}: there, two neighborly $(d-1)$-polytopes $P_1$ and $P_2$ are placed in parallel hyperplanes in $\mathbb{R}^d$ so that the Cayley polytope $P=\conv(P_1\cup P_2)$ is $\floor{\frac{d-1}{2}}$-neighborly and all $\floor{d/2}$-subsets with a vertex in $P_1$ and a vertex in $P_2$ form faces of $P$. Let $n=f_0(P)$ and $d+s=f_0(P_2)$.
Thus, for $d$ odd, any small enough perturbation of the vertices of $P_1$ into general position will change $P$, by considering the new convex hull, into $Q\in \Pl(d,n,s)$ ($(Q,P_2)$ is the ASP-pair) with $f(Q)=f(C(d,n,s))$; so $Q$ is a maximizer.
For $d$ even, in order for $Q$ to be a maximizer, we need the small perturbation be such that all $d/2$-subsets of $\ver(P_1)$ become faces of $Q$. To achieve this, we recall more from the construction of \cite{KarTza12}, and make a variation:
consider the images of the functions $\gamma_1(t,z_1,z_2)=(t,z_1 t^{d-1},t^2,t^3,\ldots,t^{d-2},z_2 t^d)\subset \mathbb{R}^d$ and
$\gamma_2(t,z_1)=(z_1 t^{d-1},t,t^2,\ldots,t^{d-2},-1)\subset \mathbb{R}^{d-1}\times\{-1\} \subset \mathbb{R}^d$.
The vertices of $P_1$ (resp. $P_2$) are on appropriate locations on the curve $\gamma_1(t,z_1^*,0)\subset \mathbb{R}^{d-1}\times\{0\}$ (resp. $\gamma_2(t,z_1^*)$, for small enough fixed $z_1^*>0$. It is possible to show, by appropriate determinant computation, that for a small enough fixed $z_2^*>0$, perturbing the vertices $\gamma_1(t_i,z_1^*,0)$ of $P_1$ to $\gamma_1(t_i,z_1^*,z_2^*)$ makes all $d/2$-subsets of $P_1$ faces of $Q$; so $Q$ is a maximizer.
\end{remark}

We proceed by producing an alternative and elementary proof of the UBT for ASP, via shelling. This will take the rest of this section.
Our proof follows ideas from the proof of the classical UBT by McMullen, cf.~\cite[Sec.8.4]{Zie95}, and from a recent work of Karavelas and Tzanaki \cite{KarTza12}. The key new ingredient is Lemma~\ref{lem:key} below, for which we need some preparation.

Let $P\in \Pl(d,n,s)$ and $(P,F)$ an ASP-pair. Let $Q$ be a polytope obtained from $P$ by stacking a new vertex $y$ beyond $F$. The $d$-polytope $Q$ is simplicial. The set of proper faces of $Q$ is the disjoint union of the faces of the complex $P':=\partial P \setminus\{F\}$ and the faces of $Q$ which contain $y$. Consequently, for all $k\ge0$,
\begin{equation}\label{eq:hQ=P+F}
h_k(Q)=h_k(P')+h_{k-1}(F).
\end{equation}

Recall the star $\st_{\mathcal{C}}(F)$ of a face $F$ in a polytopal complex $\mathcal C$ is the polytopal subcomplex generated by all the faces of $\mathcal{C}$ containing $F$.

\begin{comment}Following \cite[p.~241]{Zie95}, we say that a facet $K$ of $P$ is {\it visible} from the point $w$ if $w$ belongs to the open halfspace determined by $\aff K$ which does not meet $P$. Also, the point $w$ is {\it beyond} a face $G$ of $P$ if the facets of $P$ containing $G$ are precisely those that are visible from $w$.
\end{comment}
\begin{comment}
Next we establish a particular shelling order of the polytope $Q$, but first we need an observation from \cite{Cou06}.

\begin{lemma}{\cite[Lem.~1]{Cou06}}\label{lem:shellingOrderCou}
Let  $\mathcal{C}$ be a shellable polytopal complex and let $F_1,\ldots,F_q,H_{1},\ldots,H_p$ be a shelling order of $\mathcal{C}$. If $F'_1,\ldots,F_q'$ is a shelling order of the complex $F_1\cup\cdots\cup F_q$ then the ordering $F'_1,\ldots,F'_q,H_{1},\ldots,H_p$ is also a shelling of $\mathcal{C}$.
\end{lemma}
\end{comment}
We will use a line shelling of $Q$ with some special properties:
\begin{lemma}\label{lem:shellingOrder}
Let $(P,F)$ be an ASP-pair, and $v\in F$ a vertex.
Then we can choose the aforementioned vertex $y$ beyond $F$ such that, for $(P,F,y,Q)$ as above, there is a line shelling of $Q$ which shells the star of $y$ first and then proceeds to shell the rest of the star of $v$.  \end{lemma}

\begin{proof}
For an oriented line $\ell$ that shells $P$, with $F$ being first followed by the rest of the facets in the star of $v$ (cf.~\cite[Thm.~8.12, Cor.~8.13]{Zie95}\footnote{Here we use the extension of the notion of shellability to polyhedral complexes.}), place $y$ on $\ell$ beyond $F$ to make $Q$. Now perturb $\ell$, intersecting $Q$ near $y$, to obtain the desired line shelling.
\end{proof}

Consider any vertex $v\in \ver Q$ and let $S(Q)$ be a shelling of the facets of $Q$. Then, clearly,
\begin{itemize}
\item The restriction of $S(Q)$ to $\st_Q(v)$ yields a shelling of $\st_Q(v)$  (cf. \cite[Lem.~8.7]{Zie95}); denote it by $S_v(Q)$.
\item A shelling of $\st_Q(v)$ induces a shelling of $\lk_Q(v)$ by deleting $v$ from the facets.
\item Since $\partial F=\lk_Q(y)$, it follows that $S(Q)$ induces a shelling $S(F)$ of $F$.
\item  Recursively, $S(F)$ induces a shelling of $\lk_F(v)$ if $v\in \ver F$.
\end{itemize}

Now consider any vertex $v\in F$, a shelling $S(Q)$ as guaranteed in Lemma~\ref{lem:shellingOrder}, and the induced shellings $S(Q/v)$ of $Q/v$, $S(F)$ of $F$ and $S(F/v)$ of $F/v$. Recall that if $\st_K(v)$ is simplicial then the boundary complex of $K/v$ coincides with $\lk_{K}(v)$.

Following \cite[Sec.~4]{KarTza12}, call the facet
 $F_j$ of $Q$ \emph{active} if it is the new facet to be added to the shelling process $S(Q)$. Let $F_j|F$ be the active facet of $S(F)$ which is the restriction of $F_j$ to $F$ (if $y\in F_j$). Let $F_j/v$ be the active facet of $S(Q/v)$ induced by $F_j$ (if $v\in F_j$), $F_j|F/v$ be the active facet of $S(F/v)$ induced by $F_j|F$ (if $vy\subset F_j$), and $F_j|v$ be the active facet of $S_v(Q)$ induced by $F_j$ (if $v\in F_j$); so $F_j|v=\{v\}\cup F_j/v$ in this case.
Let $R_j\subseteq F_j$, $R_j/v\subseteq F_j/v$, $R_j|F\subseteq F_j|F$, $R_j|F/v\subseteq F_j|F/v$, and $R_j|v$ be the corresponding new minimal faces in the shellings $S(Q)$,  $S(Q/v)$, $S(F)$, $S(F/v)$, and $S_v(Q)$ respectively.

Finally, let $h_k^j(Q)$ denote the value of $h_k$ up to step $j$, namely $h_k(\cup_{i\le j}F_i)$, and similarly for the other complexes.

The following key lemma allows us to relate the difference in $h$-numbers along a shelling of $Q$ and $F$ to that of $Q/v$ and $F/v$.

\begin{lemma}\label{lem:key}
For any vertex $v\in \ver F$ and at any step $j$ of the shelling $S(Q)$ which is guaranteed by \cref{lem:shellingOrder}, we have, for all $k\ge 0$, that
\[h^j_{k}(Q)-h^j_k(Q/v) \ge  h^j_k(F)-h^j_k(F/v).\]
\end{lemma}
\begin{proof}
While shelling $\st_Q(y)$, the minimal face $R_j$ of $F_j$ in $S(Q)$ and the minimal face $R_j|F$ of $F_j|F$ in $S(F)$ coincide at every step, since $F=Q/y$ and $S(Q)$ shells the star of $y$ first.
Therefore, while shelling $\st_Q(y)$, for all $k\ge0$, it follows that \[h^j_k(Q)=h^j_k(F).\]

For the same reason, if $F_j\in \st_Q(v)\cap \st_Q(y)$, then, regardless of whether $v\in R_j$ or  $v\not\in R_j$, we have, for all $k\ge 0$, that
\[h^j_k(Q/v)=h^j_k(F/v).\]

Thus, while shelling $\st_Q(y)$, it follows for all $k\ge0$ that
\[h^j_{k}(Q)-h^j_k(Q/v)= h^j_k(F)-h^j_k(F/v).\]

After the shelling has left $\st_Q(y)$, we get no new contributions to $h_k(F)$  or $h_k(F/v)$ for all $k\ge0$, so the RHS does not change.

After shelling $\st_Q(y)$ and while still shelling $\st_Q(v)$, we have that the minimal faces $R_j$ and $R_j/v$ of $S(Q)$ and $S(Q/v)$, respectively, coincide, so the LHS does not change either.
To see that $R_j=R_j/v$, first note that $R_j/v \subseteq R_j$ (as the complex at the $j$-th step of $S(Q)$ contains the complex at the $j$-th step of $S_v(Q)$). We show the reverse containment.
Assume by contradiction that there is a facet $F"$ of $F_j$
which is in the subcomplex $\cup_{i< j}F_i$ of $Q$ but not in
the subcomplex  $\cup_{i< j,\ v\in F_i}F_i$ of $\st_Q(v)$.
As we have already left $\st_Q(y)$, $y\notin F_j$ so $F"$ is a facet of $F$.
However, also $v\in F$, so
we must have $v\in F"$, as otherwise,
by convexity, $|F_j| \subset |F|$, a contradiction.
But then the (unique) facet $F_i$ in $\st_Q(y)$ containing $F"$ is also in $\st_Q(v)$, a contradiction.

Thus, for all $k\ge 0$,
\[ h^j_{k}(Q)-h^j_k(Q/v)= h^j_k(F)-h^j_k(F/v). \]

After the shelling has left $\st(v,Q)$ we may get new contributions to $h_k(Q)$ but not any more to $h_k(Q/v)$. This concludes the proof of the lemma.
\end{proof}

\begin{proposition}\label{lem:UBTh-Aux}
Let $P\in \Pl(d,n,s)$ and $(P,F)$ be an ASP-pair.
Then, for all $k\ge 0$, we have
\begin{align*}
h_{d-(k+1)}( P')&\le\frac{n-d+k}{k+1}h_{d-k}(P')+\frac{n-(d+s)}{k+1}g_k(F).
\end{align*}
Equivalently, for all $k\ge 0$, we have
\begin{align*}
h_{k+1}( P')&\le\frac{n-d+k}{k+1}h_{k}( P')-\frac{s+k}{k+1}g_k( F)+g_{k+1}( F).
\end{align*} 	
\end{proposition}
\begin{proof}
The second inequality follows from the first by the Dehn-Sommerville relations (\cref{eq:DSEq-hP}). For the first inequality, we have the following sequence of equalities.
\begin{align}
\sum_{v\in \ver Q} h_k(Q/v)&=(k+1)h_{k+1}(Q)+(d-k)h_{k}(Q)\nonumber\\
\ &=(k+1)h_{d-(k+1)}( Q)+(d-k)h_{d-k}( Q)
\nonumber\\
\ &=(k+1)\left(h_{d-(k+1)}( P')+h_{d-(k+1)-1}( F)\right) +(d-k)h_{d-k}( P')\nonumber\\
&\quad +(d-k)h_{d-k-1}( F)
\nonumber\\
\ &=\left((k+1)h_{d-(k+1)}( P')+(d-k)h_{d-k}( P')\right)+(k+1)h_{k+1}( F)\nonumber\\
&\quad+(d-1-k)h_{k}( F)+h_{k}( F)
\nonumber\\
\ &=(k+1)h_{d-(k+1)}( P')+(d-k)h_{d-k}( P')+h_{k}(F)+\sum_{v\in \ver  F} h_k(F/v).
\label{eq:SumlkQ}
\end{align}
For the first equality, see, e.g., \cite[Eq.~8.27a]{Zie95}, while for the second, use Dehn-Sommerville \cref{eq:DSEq-hP} for $Q$. The third equality follows from \cref{eq:hQ=P+F}, the forth from \cref{eq:DSEq-hP} again, this time for $F$, and the last equality follows from \cite[Eq.~8.27a]{Zie95} again.

As $ F\cong Q/y$, \cref{eq:SumlkQ} then becomes
\begin{align}
\sum_{v\in \ver  P'\setminus \ver F} h_k(Q/v)+\sum_{v\in \ver  F} \left(h_k(Q/v)- h_k(F/v)\right)&=(k+1)h_{d-(k+1)}( P')+(d-k)h_{d-k}( P')
\label{eq:SumlkQ1}.
\end{align}
From Lemma~\ref{lem:key} and the fact that any vertex $v\in \ver  P'\setminus \ver F$ has the same link in both $Q$ and $P'$, it then follows that
\begin{align}
\sum_{v\in \ver  P'\setminus \ver F} h_k(Q/v)+\sum_{v\in \ver  F} \left(h_k(Q/v)- h_k(F/v)\right)&\le \sum_{v\in \ver  P'\setminus \ver F} h_k( P'/v)\nonumber\\
&\quad+\sum_{v\in \ver  F} \left(h_k( Q)- h_k( F)\right)
\label{eq:SumlkQ2}.
\end{align}
Let $v\in \ver  P'\setminus \ver F$. There is
a shelling of $P'$ that shells $\st_{P'}(v)$ first
-- just perturb a line through an interior point of $F$ and $v$, so it still intersects $P$ near those two points, to obtain such line shelling.
Such shelling
shows that $h_k(P'/v) \le h_k( P')$ for all $k\ge0$. Consequently, from \cref{eq:SumlkQ1,eq:SumlkQ2} it follows that
\begin{align*}
(k+1)h_{d-(k+1)}( P')+(d-k)h_{d-k}( P')&\le \sum_{v\in \ver  P'\setminus \ver F} h_k( P')+\sum_{v\in \ver  F} \left(h_k( Q)- h_k( F)\right)\\
&= (n-(d+s))h_k( P')+(d+s)(h_k( Q)- h_k( F))\\
&= nh_k( P')+(d+s)(h_{k-1}( F)- h_k( F))\\
&= n(h_{d-k}( P')+g_k( F))-(d+s)g_k( F).
\end{align*}
The last two equations ensue from \cref{eq:hQ=P+F} and \cref{eq:DSEq-hP}, respectively. Hence, the desired inequality follows.
\end{proof}	
%%%%%%%%%%%%%%%%%%%%%%%%%%
We are now in a position to prove the inequalities of Theorem~\ref{thm:UBT_ASP} using the shelling approach.

\begin{proof}[Proof of Theorem \ref{thm:UBT_ASP} via Shellings]
The inequalities \cref{eq:hk} for $P'$ in Theorem \ref{thm:UBT_ASP}  hold for any Cohen-Macaulay complex, and can also be proved exactly as in \cite[Lem.~8.26]{Zie95}. We prove \cref{eq:hd-k} by induction on $k$. The case $k=1$ holds with equality by \cref{eq:DSEq-hP}:
$h_{d-1}( P')= n-d-s$.
Suppose now that the inequality \cref{eq:hd-k} holds for $k-1\le \floor{d/2}-1$. By \cref{lem:UBTh-Aux},
\begin{align*}
h_{d-k}( P')&\le \frac{n-d+k-1}{k}\left( {n-d-1+k-1\choose k-1}-{s+k-2\choose k-1}\right)+\frac{n-(d+s)}{k}g_{k-1}( F).
\end{align*}
From the application of the $g$-theorem\footnote{In fact, this consequence follows easily just from the fact that $\partial F$ is a Cohen-Macaulay complex.}, cf. \cite[Cor.~8.38]{Zie95}, to $F$, it follows that
$g_{k-1}( F)\le {d+s-(d-1)+k-3\choose k-1}$ for $0\le k-1\le \floor{d/2}-1$. Thus, the previous inequality becomes   	
\begin{align*}
h_{d-k}( P')&\le \frac{n-d+k-1}{k}\left( {n-d-1+k-1\choose k-1}-{s+k-2\choose k-1}\right)+\frac{n-(d+s)}{k} {s+k-2\choose k-1}\\
&= {n-d-1+k\choose k} -{s+k-1\choose k},\;
\end{align*}
as desired.\end{proof}

\section*{Acknowledgments}
The authors thank the referees for their detailed comments, as well as  Karim Adiprasito, Gil Kalai and Isabella Novik for helpful feedback on an earlier version of this paper.

\providecommand{\bysame}{\leavevmode\hbox to3em{\hrulefill}\thinspace}
\providecommand{\MR}{\relax\ifhmode\unskip\space\fi MR }
\providecommand{\MRhref}[2]{%
  \href{http://www.ams.org/mathscinet-getitem?mr=#1}{#2}
}
\providecommand{\href}[2]{#2}


\begin{thebibliography}{10}

\bibitem{AdiSan14}
K.~A. Adiprasito and R.~Sanyal, \emph{Relative Stanley-Reisner theory and upper bound theorems for {M}inkowski sums}, Publications Math\'ematiques. Institut de Hautes \'Etudes Scientifiques \textbf{124} (2016), 99--163.

\bibitem{AsiRot78}
L.~Asimow and B.~Roth, \emph{The rigidity of graphs}, Trans. Amer. Math. Soc.
  \textbf{245} (1978), 279--289. \MR{511410 (80i:57004a)}


\bibitem{Asi-Roth2}
\bysame,
     \emph{The rigidity of graphs. {II}},
   J. Math. Anal. Appl.,
   \textbf{68} (1979), 171--190.


\bibitem{Bar71}
D.~W. Barnette, \emph{The minimum number of vertices of a simple polytope},
  Israel J. Math. \textbf{10} (1971), 121--125. \MR{0298553 (45 \#7605)}

\bibitem{Bar73}
\bysame, \emph{A proof of the lower bound conjecture for convex polytopes},
  Pacific J. Math. \textbf{46} (1973), 349--354. \MR{0328773 (48 \#7115)}

\bibitem{Billera-Lee}
L.~J. Billera and C.~W. Lee,
     \emph{A proof of the sufficiency of {M}c{M}ullen's conditions for
              {$f$}-vectors of simplicial convex polytopes},
   J. Combin. Theory Ser. A,
   \textbf{31},
      (1981),
     237--255.

\bibitem{BilLee81a}
\bysame, \emph{The numbers of faces of polytope pairs and
  unbounded polyhedra}, European J. Combin. \textbf{2} (1981), no.~4, 307--322.
  \MR{638405 (82m:52005)}

\bibitem{Bra06}
T.~Braden, \emph{Remarks on the combinatorial intersection cohomology of fans},
  Pure Appl. Math. Q. \textbf{2} (2006), no.~4, part 2, 1149--1186. \MR{2282417
  (2008c:14031)}

\bibitem{BraMac99}
T.~Braden and R.~D. MacPherson, \emph{Intersection homology of toric varieties
  and a conjecture of {K}alai}, Comment. Math. Helv. \textbf{74} (1999), no.~3,
  442--455. \MR{1710686 (2000h:14018)}

\bibitem{Bro83}
A.~Br{\o}ndsted, \emph{An introduction to convex polytopes}, Graduate Texts in Mathematics, vol.~90, Springer-Verlag, New York, 1983. \MR{683612 (84d:52009)}

\bibitem{BruMan71}
H.~Bruggesser and P.~Mani, \emph{Shellable decompositions of cells and
  spheres}, Math. Scand. \textbf{29} (1971), 197--205 (1972). \MR{0328944 (48
  \#7286)}



\bibitem{GooORo04}
J.~E. Goodman and J.~O'Rourke (eds.), \emph{Handbook of discrete and
  computational geometry}, 2nd ed., Discrete Mathematics and its Applications
  (Boca Raton), Chapman \& Hall/CRC, Boca Raton, FL, 2004. \MR{2082993
  (2005j:52001)}

\bibitem{Gru03}
B.~Gr{{\"u}}nbaum, \emph{Convex polytopes}, 2nd ed., Graduate Texts in
  Mathematics, vol. 221, Springer-Verlag, New York, 2003, Prepared and with a
  preface by V. Kaibel, V. Klee and G. M. Ziegler. \MR{1976856 (2004b:52001)}

\bibitem{Kal87}
G.~Kalai, \emph{Rigidity and the lower bound theorem. {I}}, Invent. Math.
  \textbf{88} (1987), no.~1, 125--151. \MR{877009 (88b:52014)}

\bibitem{Kal94}
\bysame, \emph{Some aspects of the combinatorial theory of convex polytopes},
  Polytopes: abstract, convex and computational ({S}carborough, {ON}, 1993),
  NATO Adv. Sci. Inst. Ser. C Math. Phys. Sci., vol. 440, Kluwer Acad. Publ.,
  Dordrecht, 1994, pp.~205--229. \MR{1322063 (96b:52018)}

\bibitem{KarTza12}
M.~I. Karavelas and E.~Tzanaki, \emph{The maximum number of faces of the
  {M}inkowski sum of two convex polytopes}, Proceedings of the {T}wenty-{T}hird
  {A}nnual {ACM}-{SIAM} {S}ymposium on {D}iscrete {A}lgorithms, ACM, New York,
  2012, pp.~11--22. \MR{3205193}

\bibitem{Kar06}
K.~Karu, \emph{Hard {L}efschetz theorem for nonrational polytopes},
   Invent. Math., \textbf{157} (2004), 419--447.

\bibitem{Kolins11}
S.~Kolins, \emph{$f$-vectors of triangulated balls},
Disc. Comp. Geom., \textbf{46} (2011), 427--446.

\bibitem{McM70a}
P.~McMullen, \emph{The maximum numbers of faces of a convex polytope},
  Mathematika \textbf{17} (1970), 179--184. \MR{0283691 (44 \#921)}

\bibitem{McMullen-g-conj}
\bysame,
     \emph{The numbers of faces of simplicial polytopes},
   Israel J. Math.,
   \textbf{9},
   1971,
   559--570.

\bibitem{McMWal71}
P.~McMullen and D.~W. Walkup, \emph{A generalized lower-bound conjecture for
  simplicial polytopes}, Mathematika \textbf{18} (1971), 264--273. \MR{0298557
  (45 \#7609)}

\bibitem{MurNev13}
S.~Murai and E.~Nevo, \emph{On the generalized lower bound conjecture for
  polytopes and spheres}, Acta Math. \textbf{210} (2013), no.~1, 185--202.
  \MR{3037614}

\bibitem{Nevo-Novinsky}
E.~Nevo and E.~Novinsky,
     \emph{A characterization of simplicial polytopes with {$g_2=1$}},
   J. Combin. Theory Ser. A,
    \textbf{118},
      (2011) 387--395.


\bibitem{Sta80}
R.~P. Stanley, \emph{The number of faces of a simplicial convex polytope},
  Adv. Math. \textbf{35} (1980), 236--238.

\bibitem{Stanley:GeneralizedH-vectors}
\bysame,
     \emph{Generalized {$H$}-vectors, intersection cohomology of toric
              varieties, and related results},
Commutative algebra and combinatorics (Kyoto, 1985),
    {Adv. Stud. Pure Math.},
    \textbf{11}, 187--213,
North-Holland, Amsterdam, 1987.

\bibitem{Whi84}
W.~Whiteley, \emph{Infinitesimally rigid polyhedra. {I}. {S}tatics of
  frameworks}, Trans. Amer. Math. Soc. \textbf{285} (1984), no.~2, 431--465.
  \MR{752486 (86c:52010)}

\bibitem{Zie95}
G.~M. Ziegler, \emph{Lectures on polytopes}, Graduate Texts in Mathematics,
  vol. 152, Springer-Verlag, New York, 1995. \MR{1311028 (96a:52011)}

\end{thebibliography}
\end{document}